\documentclass{amsart}
\usepackage[all, cmtip]{xy}
\usepackage{tikz-cd} 
\usepackage{color}
\usepackage{soul}
\usepackage{url}
\usepackage{hyperref}
\usepackage{mathtools}

\newcommand{\ground}{\mathbf{k}}
\newcommand{\rR}{\mathbb{R}}

\DeclareMathOperator{\divergence}{\nabla \cdot}
\DeclareMathOperator{\curl}{\nabla \times} 
\newcommand{\ip}{\cdot}

\DeclareMathOperator{\e}{E}

\newtheorem{theorem}{Theorem}[section]
\newtheorem*{FTHPT}{The Fundamental Theorem of HPT}

\newtheorem{lemma}[theorem]{Lemma}

\theoremstyle{definition}
\newtheorem{definition}[theorem]{Definition}
\newtheorem{example}[theorem]{Example}
\theoremstyle{remark}
\newtheorem*{remark}{Remark}
\newtheorem*{question*}{Question}

\title[Homotopy Probability on a Manifold and the Euler Equation]{Homotopy Probability Theory on a Riemannian manifold and the Euler equation}

\author{Gabriel C. Drummond-Cole}
\address{Center for Geometry and Physics, 
Institute for Basic Science (IBS), Pohang, Republic of Korea 37673}
\email{gabriel@ibs.re.kr}
\thanks{This work was supported by IBS-R003-D1}

\author{John Terilla}
\address{Department of Mathematics, The Graduate Center and Queens
  College, The City University of New
  York, USA}
\email{jterilla@gc.cuny.edu}

 \keywords{probability, fluids, Riemannian manifolds, homotopy}

 \subjclass[2010]{55U35, 58Axx, 58Cxx, 60Axx, 76xx}

\begin{document}

\begin{abstract}
Homotopy probability theory is a version of probability theory in which the vector space of random variables is replaced with a chain complex.  
A natural example extends ordinary probability theory on 
a finite volume Riemannian manifold $M$. 

In this example, initial
conditions for fluid flow on $M$ are identified with collections of 
homotopy random variables and solutions to the Euler equation
are identified with homotopies between collections of homotopy
random variables. 

Several ideas about using homotopy probability theory to study fluid flow are introduced.

\end{abstract}

\maketitle
\setcounter{tocdepth}{1}
\tableofcontents 

\section{Introduction} 
Homotopy probability theory is a version of probability theory in which the vector space of 
random variables is replaced with a chain complex.  The ordinary probability theory on a Riemannian manifold, in which the random variables are measurable functions, sits within a larger homotopy probability theory in which collections of random variables may include differential forms.  Collections of differential forms should satisfy certain conditions in order to have meaningful, homotopy invariant statistics.  Collections of functions and forms that satisfy those conditions constitute collections of \emph{homotopy random variables}.  The conditions can be summarized by saying that the moment generating function of the collection  is closed with respect to the adjoint of the de Rham differential.  To define a moment generating function, parameters are chosen to keep track of the joint moments of the constituents of the collection.  In this paper, particular parameter rings are considered and the collections of homotopy random variables are identified with initial conditions for fluid flows, such as an initial density.  Solutions to fluid flow equations are identified with homotopies between collections of homotopy random variables parametrized by particular rings. In this regime the homotopy parameter plays the role of time.  The particular choice of fluid flow equation that is applicable depends on the choice of parameter ring. In this paper, the mass equation and the Euler equation appear.

The homotopy probability framework provides some new ideas for studying fluid flow and a couple of these ideas are introduced in Section \ref{section: applications} after reviewing homotopy probability theory on a Riemannian manifold and explaining the connection with fluids.  One 
idea is to study the homotopy invariant statistics afforded by homotopy 
probability theory which necessarily yield hierarchies of time-independent 
invariants of fluid flow.  A second idea is to use a finite cochain 
model for a manifold giving rise to a homotopy probability space 
quasi-isomorphic to the smooth one relevant for fluid flow.  In the 
combinatorial model, homotopies between combinatorial collections of homotopy random variables are solutions to finite dimensional ODEs and can naturally be transported to give homotopies between smooth collections of homotopy random variables.  Not all homotopies between collections of homotopy random variables constitute solutions to fluid flow equations.  This paper contains precise descriptions of the additional conditions that homotopies must satisfy in order to be solutions to the fluid equations.  These conditions and their combinatorial versions could be studied to get insights into smooth fluids, or could be used to develop new computer models.

The authors are grateful to Jae-Suk Park, Dennis Sullivan, and Scott Wilson for many useful discussions. 

\section{Ordinary probability theory on a Riemannian manifold}\label{section:rm}
Let $M$ be a connected, closed, oriented $n$-dimensional Riemannian manifold and let $dV$ be the associated volume form on $M$. Assume the metric is normalized so that $\int_M dV 
=1$. 
There is an ordinary probability space associated to $M$. 
The bounded measurable functions on $M$ 
are the random variables and the expectation is defined 
by 
\[\e(f)=\int_M f dV.\] 
For simplicity, restrict the random variables to $\Omega^0(M)$, the smooth 
functions on $M$. Note that the 
expectation factors 
\[
\xymatrix{\Omega^0(M)\ar[dr]_{\star}\ar[rr]^{\e} 
&& 
\rR 
\\ 
& \Omega^n(M)\ar[ur]_{\int_M} 
}\]
where the $n$-form $\star(f)$ is defined as $fdV$.

As a linear map to a one-dimensional vector space, 
the expectation $E:\Omega^0(M)\to \rR$ can be understood by 
studying its kernel. 
Stokes' theorem implies that if
$fdV=d\omega$ for an $(n-1)$-form 
$\omega$ then 
\[E(f)=\int_M fdV=\int_M d\omega = 0.\] The converse also holds because $H^n(M)$ is one dimensional, spanned by the cohomology class of $dV$. So understanding the kernel 
of the expectation amounts to understanding which functions $f$ have 
the property that $fdV$ is exact. 

The map $\star$ is part of the 
so called ``Hodge star'' isomorphism $\star:\Omega^\bullet(M)\to 
\Omega^{n-\bullet}(M)$. One can use the isomorphism $\star$ to transport the differential $d$ 
to the codifferential $\delta= \pm \star^{-1} d \star$ (see Appendix~\ref{appendix:musical appendix} for details). 
Then, in the following diagram 
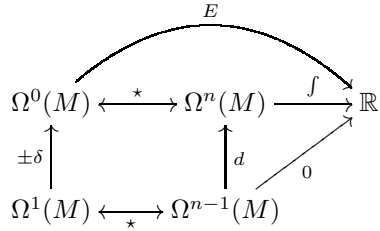
\[
\xymatrix{
\Omega^0(M)\ar@{<->}[r]^-\star \ar@/^2.5pc/[rr]^E &\Omega^n(M)\ar[r]^-\int & \rR
\\
\Omega^1(M)\ar@{<->}[r]_-\star\ar[u]^{\pm \delta}&
\Omega^{n-1}(M)\ar[u]_{d}\ar[ur]_{0}
}
\]

% \begin{center} 
% \begin{tikzcd} 
% \Omega^0(M)\ar[<->]{r}{\star} \arrow[rr,bend left=45,"E"]& \Omega^n(M)\ar{r}{\int}& \rR 
% \\ 
% \Omega^1(M)\ar[<->]{r}[swap]{\star}\ar{u}{{\pm\delta}}& 
% \Omega^{n-1}(M)\ar{u}[swap]{d}\ar{ur}[swap]{0} 
% \end{tikzcd} 
% \end{center} 
the kernel of $E$ is precisely the image of 
${\delta}:\Omega^{1}(M)\to \Omega^0(M)$.

Now, extend the expectation to a 
map on all forms by setting $E(\alpha)=0$ for all 
$k$-forms $\alpha$ with $k\ge 1$. Regard the real 
numbers as a chain complex concentrated in degree zero with zero 
differential. Then the situation can be encapsulated in the statement  
\begin{quote}\emph{Expectation is a chain map 
$(\Omega^\bullet(M),{\delta}) \overset{E}{\longrightarrow} ( \rR,0)$ which induces an isomorphism in cohomology.} 
\end{quote} 
If one were only interested in 
expectations, this 
would be the final statement in this line of thinking. However, in probability 
theory, one is also interested in correlations among random 
variables. That is, the space of random variables is an algebra and 
so is the set of real numbers but expectation is not an 
algebra map. On the contrary, the failure of expectation to be an 
algebra map measures the extent to which random variables are 
statistically dependent. For 
example, the difference \[E(fg)-E(f)E(g)\] is a particular measure of 
correlation called the \emph{covariance of $f$ and $g$}. Note too 
that ${\delta}$ fails to be compatible with the algebra structure---the Hodge star operator does not respect the wedge product. Section \ref{section: homotopy random variables and homotopy statistics} explains how to combine the ways that ${\delta}$ and $E$ fail to respect the algebra structures to define a homotopy theory of statistical dependence.  But first, some of the basic ideas and terminology of homotopy probability theory are reviewed.

\section{Homotopy probability theory on a Riemannian manifold}\label{section:hpt} 

Homotopy probability spaces, studied before in~\cite{Park:L, DrummondColeParkTerilla:HPTII, ParkPark:EHTPISPH, DrummondColeTerilla:CIHPT,Park:HTPSICIHLA}, are now briefly recalled. There is a ground field (always $\rR$ in this paper) which is viewed as a chain complex concentrated in degree zero with zero differential. Recall that a chain complex is pointed if it is equipped with a chain map $1$ from the ground field, called the unit (so $d1=0$). Morphisms of pointed chain complexes are required to preserve the unit. 
\begin{definition} 
A \emph{homotopy pre-probability space} is a tuple $(V,1, d,m)$ where $(V,1,d)$ is a pointed chain complex and $(V,m)$ is a graded commutative associative algebra for which $1$ is a unit for $m$. A morphism of homotopy pre-probability spaces is a morphism of pointed chain complexes. 
\end{definition} 
Note that other than the unital condition, there is no assumed compatibility between the product and either the differential or a morphism of homotopy pre-probability spaces. A homotopy pre-probability space should be thought of as a generalization of the space of random variables. In order to define expectations and correlations, one needs an expectation map and a space for the expectation values to live in. 
\begin{definition} 
Let $\ground$ be a homotopy pre-probability space. A {\em $\ground$-valued homotopy probability space} is a homotopy pre-probability space $C$ equipped with a morphism $E$ of homotopy pre-probability spaces $E:C\to \ground$, called \emph{expectation}. A morphism of $\ground$-valued homotopy probability spaces is a morphism of homotopy pre-probability spaces that respects expectation. 
\end{definition} 
\begin{definition} 
An \emph{ordinary probability space} is an $\rR$-valued homotopy probability space concentrated in degree $0$. 
\end{definition} 

Now the motivating example from the previous section is formalized. 

\begin{definition} 
Let $M$ be a closed Riemannian manifold. Then $(\Omega^\bullet(M),1,{\delta},\wedge)$ is a homotopy pre-probability space called the \emph{homotopy pre-probability space associated to $M$.} 
Integration of functions against the volume form $E:\Omega^\bullet(M) \to \rR$ makes $(\Omega^\bullet(M),1,{\delta},\wedge)$ into an $\rR$-valued homotopy probability space called the ($\rR$-valued) \emph{homotopy probability space associated to $M$.} Restricting to the smooth functions $\Omega^0(M)$ yields the ordinary probability space associated to $M$. 
\end{definition} 

There are multiple natural generalizations of the homotopy probability space associated to a Riemannian manifold described above. 

\subsection{Cohomology-valued homotopy probability space} 

One variation that is possible is to replace the real numbers with the cohomology of the manifold. Let $M$ be as above. Then the real cohomology $H^\bullet(M)\coloneqq H^\bullet(M,\rR)$ is a unital graded commutative associative algebra. There is a canonical isomorphism of pointed graded vector spaces between $H^\bullet(M)$ and the harmonic forms 
\[\mathcal{H}^\bullet(M)=\{\alpha \in \Omega^\bullet(M): d\alpha={\delta}\alpha=0\}.\] 

The Riemannian metric on $M$ gives rise to a bilinear form on $\Omega^\bullet(M)$ over the ring of functions on $M$ and there is a natural orthogonal projection $\Omega^\bullet(M)\to \mathcal{H}^\bullet(M)$, which gives a map $E:\Omega^\bullet(M) \to H^\bullet(M)$ which is independent of basis, but can be described using a basis $\{e_i\}$ of $\mathcal{H}^\bullet(M)$ by 
\[ 
\omega\mapsto \text{ the class of }\sum_i \left(\int_M \omega\wedge \star e_i\right)e_i. 
\] 
Then $(\Omega^\bullet(M), 1, {\delta}, \wedge)$ and $(H^\bullet(M), 1, 0, \smile)$ are homotopy pre-probability spaces and $E:\Omega^\bullet(M)\to H^\bullet(M)$ makes $\Omega^\bullet(M)$ into an $H^\bullet(M)$-valued homotopy probability space, called the $H^\bullet$-valued homotopy probability space associated to $M$.

\subsection{Noncompact manifolds} 
One may want to work with a pre-probability space associated to a noncompact Riemannian manifold. In this case, some attention must be paid to the type of forms that are used~\cite{ArnoldFalkWinther:FEECHTA}. The following data and properties are desirable for the main ideas and results in HPT to apply. 
\begin{itemize} 
\item Spaces $W^k$ of differential $k$-forms with a linear expectation map $E$ defined on the functions $W^0$ 
\item For unitality, the constant function $1$ should be in $W^0$ and have expectation $E(1)=1$.
\item A degree $-1$ differential $\delta$ with $\ker(E)=\delta(W^1)$ 
\item A degree zero algebra structure on the sum $W=\bigoplus_k W^k$ 
\end{itemize} 
For an oriented Riemannian manifold $M$, there may be numerous choices for the spaces of differential forms with the desired properties. If $M$ has finite volume, uniformly bounded smooth differential forms are one choice. If one is indifferent about the resulting homotopy probability space being unital, compactly supported smooth forms would also work, whether $M$ is finite volume or not. Particular circumstances, as in example \ref{homotopygaussian} below where polynomial forms are defined, suggest other choices are possible. An idea expressed by Terence Tao in a blog post about algebraic probability theory~\cite{Tao:254AN5FP} involved functions whose $L^p$ norms are bounded for all $p$---these spaces can probably be extended to some sort of forms in order to do homotopy probability theory on a noncompact Riemannian manifold. 

There are other considerations, not pursued here, for the important case of manifolds with boundary.
\begin{example}\label{homotopygaussian} 
The homotopy Gaussian, considered in other contexts as well \cite{GwilliamJohnsonFreyd:HDFDFDIDBVF,ParkPark:EHTPISPH, DrummondColeTerilla:CIHPT}, is an example of a homotopy probability space associated to a noncompact finite-volume Riemannian manifold. Consider the ordinary Gaussian measure 
$d\mu=(2 \pi)^{-1/2}e^{-x^2/2}dx$ on $M=\rR$ and let $W^\bullet(\rR)$ consist of differential forms of the form $\rR[x]\langle 1, \eta\rangle$ where $\eta =(2\pi)^{-1}e^{-x^2} dx$. For any function $f\in W^0(\rR)$, one computes that \[\star f = f\, d\mu\text{ and }\star (f\, d\mu) = f\] and finds $\delta:W^1(\rR) \to W^0(\rR)$ is given by $\delta (f\eta)=f'-xf$. 

Thus $(W^\bullet(\rR),\delta )$ with the expectation map defined by $E(f)=\int\! f\, d\mu$ is a unital homotopy probability space in which probability computations can be performed with simple algebra,. For example, the moments of $f$ can be readily computed by knowing that $E$ vanishes on $\delta$ exact forms. The first two examples of such calculations are
\[0=E(\delta \eta)=E(x)\] and 
\[0=E(\delta x\eta)=E(x^2-1)\] so that $E(x^2)=E(1)=1$.  This example is explored in more detail in the references above.
\end{example} 

\section{Homotopy statistics}
\label{section: homotopy random variables and homotopy statistics}
Gromov recently wrote \cite{Gromov:SPESLM2014}
\begin{quote}{\em The success of probability theory decisively,
  albeit often invisibly, depends on symmetries of the system the
  theory applies to.}
\end{quote}
The example of the homotopy Gaussian and more generally the example of forms on a finite volume Riemannian manifold illustrate a method to organize symmetries of an ordinary probability theory when those symmetries arise geometrically.  The method is to generate a related homotopy probability space, and then use algebra to compute expectations.  Computations performed in the related homotopy probability space should genuinely relate to the original ordinary probability space and so they should be homotopy invariant in an appropriate sense---homotopic expectation morphisms should yield the same result on random variables, and homotopic collections of random variables should have the same homotopy statistics.  

Therefore, unlike in ordinary probability theory in which every element of $V$ is considered as a random variable, the notion of a random variable in a homotopy probability space is refined.  The starting point for defining collections of homotopy random variables and their homotopy statistics is that an element $X$ of $V$ has a meaningful expectation if and only if $dX=0$.  Random variables could be defined as closed elements of a homotopy probability space if one were only interested in expectation values.  However, this first guess of a definition of random variables is unsatisfactory from the point of view of other correlations and other statistics, such as covariance, because there is no compatibility assumed between $d$ and the product.  So, for example, $dX=0$ and $dY=0$ does not imply that $d(XY)=0$ and therefore the joint moment $E(XY)$ and the covariance $E(XY)-E(X)E(Y)$ will be meaningless.

In \cite{DrummondColeParkTerilla:HPTII,DrummondColeTerilla:CIHPT}, collections of homotopy random variables were defined using the technology of $L_\infty$ algebras and $L_\infty$ morphisms.  In this paper, collections of homotopy random variables are reformulated.  The reformulation is potentially more accessible and more in line with the language of probability theory, but only makes sense in the presence of an exponential function.  The exponential reformulation of homotopy random variables works perfectly well for homotopy probability theory on a Riemannian manifold since the infinite sums defining the exponential function all converge.  In other settings there would be algebraic technicalities to address.  One way (possibly overkill) to address the technicalities is to assume every ring in sight is a complete topological ring satisfying the so-called Arens or $m$-convexity condition~\cite{Arens:TSLOCTR}.  Such issues are not germane here and so the reader is invited to interpret all the definitions in this section as applying only for the homotopy pre-probability space $(\Omega^\bullet(M),1,\delta,\wedge)$ associated to a Riemannian manifold $M$ or to any setting in which the exponential function makes sense.

\subsection{Collections of homotopy random variables}
Let $(V,d,m)$ be a homotopy pre-probability space and consider $V[[s]]$, where $s$ is a formal parameter.  Extend the differential $d$ and the product $m$ linearly over $s$.  Then the moment generating function 
\[\exp(sX)=1+sX+\frac{s^2}{2}X^2+\frac{s^3}{3!}X^3+\cdots \] 
satisfies
$d\left (\exp\left (sX\right)\right)=0$ if and only if $d(X^n)=0$ for all $n$.  That is, $\exp(s X)$ is closed if and only if all the moments of $X$ are meaningful.  If $X$ has nonzero degree, then it is convenient to allow the formal parameter $s$ to have a nonzero degree so that $sX$ has degree $0$.  More generally, all the joint moments of members of a collection $\{X_i\}_{i\in \Lambda} \subset V$ will be well defined if and only if their joint moment generating function is closed: 
\[d\left (\exp\left (\sum_{i\in \Lambda} s^i X_i\right )\right)=0 \text{ in }V[[\{s_i\}_{i\in \Lambda}]]\]
since every finite joint moment of the $\{X_i\}$ appears as a single term in the expression $\exp\left (\sum_{i\in \Lambda} s^i X_i\right )$ with a unique formal parameter as a coefficient. 

The parameter ring $V[[\{s_i\}_{i\in \Lambda}]]$ is isomorphic to $V\otimes k[[\{s_i\}_{i \in \Lambda}]]$ but there are reasons to allow the ring parametrizing a collection of random variables to be more general than $k[[\{s_i\}_{i\in \Lambda}]]$.  For example, if one is interested in joint moments only up to a certain point, then it is only necessary that $d(\exp(\sum_{i\in \Lambda} s^i X_i))=0$ in $V \otimes R$ where $R$ is a quotient of the parameter ring $k[[\{s_i\}_{i \in \Lambda}]]$.  One potential problem is that quotients are often poorly behaved from a homotopical point of view. One solution is to replace such quotients with free resolutions, which requires moving from graded rings to differential graded rings. This is merely one aspect of the motivation behind using differential graded rings---neither the requisite homotopy theory nor any other motivating ideas are dwelled upon in this paper.

\begin{definition}\label{definition: HRV}
Let $R=(\mathbb{R}[[S]],d_R)$ be a power series ring on a set $S$ of graded variables equipped with a differential and let $(V,1,d,m)$ be a homotopy pre-probability space.
A \emph{collection of homotopy random variables parametrized by $R$} is a degree zero element $\mathcal{X}\in V\otimes R$ satisfying $\delta \left (\exp\left (\mathcal{X}\right )\right )=0$ and the \emph{homotopy statistics} of such a collection is defined to be $E(\exp(\mathcal{X}))\in R$.  
\end{definition}

The $\delta$ in the expression $\delta(\exp(\mathcal{X})=0$ means the tensor product differential for the tensor product $\Omega^\bullet(M)\otimes R$.  That is, if $\delta_M$ is the differential in $\Omega^\bullet(M)$ and $d_R$ is the differential in $R$, $\delta=\left(\delta_M\otimes 1 + 1\otimes d_R \right)$.  
Also, $E:V \to \rR$ is extended to a map $V\otimes R \to \rR \otimes R \simeq R$.

\subsection{Homotopies of collections of homotopy random variables}
A parameter ring $R$ can be extended using an algebraic model for the interval to $R[[t,dt]]$.  By convention $t$ has degree zero, $dt$ has degree $-1$, $d(t)=dt$ and $d(dt)=0$.  

\begin{definition}\label{homotopy}
Two collections $\mathcal{X}_0$ and $\mathcal{X}_1$ of homotopy random variables parametrized by a parameter ring $R$ are \emph{homotopic} if and only if there exists an $R[[t,dt]]$-collection $\mathcal{X}$ of homotopy random variables satisfying $\mathcal{X}\vert_{t=0,dt=0}=\mathcal{X}_0$ and $\mathcal{X}\vert_{t=1,dt=0}=\mathcal{X}_1$.  Such a collection $\mathcal{X}\in \Omega^\bullet\otimes R[[t,dt]]$ is called a \emph{homotopy} between the collections $\mathcal{X}_0$ and $\mathcal{X}_1$.  
\end{definition}

Homotopic random variables can be characterized as having cohomologous exponentials.

 \begin{lemma}\label{HomotopyLemma}
 Two collections of homotopy random variables $\mathcal{X}_0$ and 
 $\mathcal{X}_1$ parametrized by $R$ are homotopic if and only if $\exp(\mathcal{X}_0)-\exp(\mathcal{X}_1)=\delta \mathcal{Y}$ for some collection $\mathcal{Y}\in \Omega^\bullet(M)\otimes R$.
 \end{lemma}
\begin{proof}Let $\mathcal{X}_0 $ and $\mathcal{X}_1$ in  $\Omega^\bullet(M)\otimes R$ be homotopic. This means there exists $\mathcal{X}=A(t)+B(t)dt \in \Omega^\bullet(M)\otimes R\otimes \rR[[t,dt]]$ so that the evaluation $A(0)=\mathcal{X}_0$, the evaluation $A(1)=\mathcal{X}_1$ and $\delta(\exp(A(t)+B(t)dt)=0$. This final condition breaks into a $t$ term and a $dt$ term. The $dt$ term of $\delta(\exp(A(t)+B(t)dt)=0$ is the equation 
\[\left(\frac{\partial\exp(A(t))}{\partial t}+\delta(B(t)\exp(A(t)))\right) dt=0\] and integrating from $t=0$ to $t=1$ yields 
\[\exp(A(1))-\exp(A(0))+\delta \int_0^1 B(t)\exp(A(t))dt =0\]
which implies that $\exp(\mathcal{X}_1)-\exp(\mathcal{X}_0)$ is exact.

To prove the converse, 
suppose $\exp(\mathcal{X}_1)-\exp(\mathcal{X}_0)=\delta(\mathcal{Y})$ and $\delta(\exp(\mathcal{X}_i))=0$ for $i=0,1$. Consider the following element in $\Omega^\bullet(M)\otimes R\otimes \rR[[t,dt]]$:
\begin{multline*}
\mathcal{X}_t = \log\left(\exp(\mathcal{X}_0)(1-t)+\exp(\mathcal{X}_1)t\right) - \frac{\mathcal{Y}}{\exp(\mathcal{X}_0)(1-t)+\exp(\mathcal{X}_1)t}dt.
\end{multline*}
By inspection, evaluating at $t=0,dt=0$ and $t=1,dt=0$ yield, respectively, $\mathcal{X}_0$ and $\mathcal{X}_1$. Then
\begin{multline*}
\exp(\mathcal{X}_t)=\left(\exp(\mathcal{X}_0)(1-t)+\exp(\mathcal{X}_1)t\right)\left(
1-
\frac{\mathcal{Y}dt}{\exp(\mathcal{X}_0)(1-t)+\exp(\mathcal{X}_1)t}
\right)
\\=
\exp(\mathcal{X}_0)(1-t)+\exp(\mathcal{X}_1)t - \mathcal{Y}dt
\end{multline*}
and thus
\[
\delta(\exp(\mathcal{X}_t))=
\delta(\exp(\mathcal{X}_0))(1-t)-\exp(\mathcal{X}_0)dt+\delta(\exp(\mathcal{X}_1))t + \exp(\mathcal{X}_1)dt-\delta(\mathcal{Y})
\]
which vanishes by the assumptions on $\mathcal{X}_0$, $\mathcal{X}_1$, and $\mathcal{Y}$.

The following equivalent definition shows that as long as the exponentials make sense there is no issue with convergence.
\begin{multline*}
\mathcal{X}_t = \mathcal{X}_0 -\sum_{n=1}^\infty\frac{1}{n}(1-\exp(\mathcal{X}_1-\mathcal{X}_0))^nt^n + \exp(-\mathcal{X}_0)\mathcal{Y}\sum_{n=0}^\infty(1-\exp(\mathcal{X}_1-\mathcal{X}_0))^nt^ndt.
\end{multline*}
\end{proof}
This proof is naive, but the Lemma is true for conceptual reasons and in much more generality. The reference~\cite{DotsenkoPoncin:TTH} provides an excellent overview.

\begin{FTHPT}The statistics of homotopic collections of homotopy random variables parametrized by a ring $R$ have the same cohomology class in $R$.
\end{FTHPT}\label{thm:fundamental theorem}
\begin{proof}$E$ is a chain map so has the same value on elements with the same cohomology class.  Lemma \ref{HomotopyLemma} says that the joint moment generating functions of homotopic collections of homotopy random variables have the same cohomology class.
\end{proof}

\section{Homotopy random variables on a Riemannian manifold}\label{section:hrv on a Riemannian manifold}
The main idea of this paper is that solutions to fluid flow equations can be identified with homotopies of parametrized collections of homotopy random variables.  The homotopy parameter plays the role of time.  The other parameters keep track of various quantities related to the fluid such as velocity and density.  The fluid flow equations guarantee that the parametrized quantities satisfy the condition $\delta \exp(\mathcal{X})=0$ required for homotopies of random variables.  

The choice of parameter ring is delicate.  For different parameter rings, the general solution to $\delta \exp(\mathcal{X})=0$ decomposes into different equations and thus different rings are paired with different fluid flow equations.   For each parameter ring presented here, there is an injection
\[\left \{ \parbox{3cm}{\centering Solutions to fluid flow equations} \right\} \hookrightarrow \left\{ \parbox{4.5cm}{\centering Homotopies of collections of homotopy random variables}\right\}\]
The map is not a bijection and this section describes its image by giving explicit conditions and choices which imply that a homotopy of parameterized collections of homotopy random variables comes from a solution to the fluid flow equations.  Each of the several parameter rings discussed has its own strengths and weaknesses, as can be seen below. It is an interesting puzzle to find better parameter rings with more strengths and fewer weaknesses.

Various diffeo-geometric notions are used in this section for decomposing the equation 
$\delta \exp \sum s_i \omega_i=0$.  These notions, including several operators and their compatabilities, are reviewed in the appendix.

\subsection{The mass equation}
To begin with, look at a single ordinary random variable, that is, a collection of homotopy random variables parametrized by the $1$-dimensional parameter ring $R=\rR$.  Such a collection is a degree zero element $f$ in $\Omega^\bullet(M)$, that is, a function, satisfying $\delta \exp(f)=0$.  This equation is vacuous for a function---all collections of ordinary random variables are collections of homotopy random variables.  So a collection of homotopy random variables parametrized by $\rR$ is a function. In the sequel, it will be convenient to use the shorthand notation $\rho$ for $\exp(f)$.

Now look at a homotopy.  A homotopy of random variables parametrized by $\rR$ is a degree zero element $f(t)+X(t)dt$ in the homotopy pre-probability space $\Omega^\bullet(M)[[t,dt]]$. Here $f$ is a power series in $t$ whose coefficients are functions on $M$ and $X$ is a power series in $t$ whose coefficients are one-forms in $M$.  The series $f$ and $X$ must satisfy the equation 
\[
\delta \exp(f+Xdt)=0.
\]
Because $dt$ squares to zero, this exponential simplifies and we get
\[
\delta(\rho  + \rho Xdt) = 0
\]
which polarizes into the two expressions
\begin{align}
\delta \rho &=0\nonumber
\\
\dot\rho + \delta (\rho  X)&=0\label{MCmasseq}
\refstepcounter{equation}\tag{\theequation a}.
\end{align}
The unlabeled equation is again vacuous. Equation~(\ref{MCmasseq}) is not.

\begin{lemma}\label{lemma:masseq}
Solutions to Equation~(\ref{MCmasseq}) are in bijective correspondence to solutions to the \emph{mass equation}
\begin{align}
\dot{\rho}+\divergence{(\rho u)}&=0,\label{ordinarymasseq}\tag{\theequation b}
\end{align}
where $\rho$ is a positive \emph{density} function and $u$ a \emph{velocity} vector field on $M$, via the identification of the two $\rho$ variables and 
\begin{align*}
u &= X^\sharp.
\end{align*}
\end{lemma}
\begin{proof}
The flat and sharp isomorphisms are linear over functions and interchange the divergence operator and $\delta$ on one-forms.
\end{proof}
Lemma~\ref{lemma:masseq} justifies referring to either Equation~(\ref{MCmasseq}) or Equation~(\ref{ordinarymasseq}) interchangably as the mass equation. 

One can reinterpret what homotopy probability theory says about the ordinary probability theory on $M$.  The exponential of any function is a positive function on $M$ which can be thought of as a density function.  Integrating density over the manifold computes mass.  Suppose that $f_0$ and $f_1$ are two functions on $M$ that give rise to the same mass \[\int_M e^{f_0} dV = \int_M e^{f_1} dV\] which on a connected manifold implies that the difference of densities is $\delta$-exact $e^{f_0}-e^{f_1}=\delta\alpha$.  Therefore, by Lemma \ref{HomotopyLemma} the functions $f_0$ and $f_1$ are homotopic collections of homotopy random variables.  So, there exists a homotopy $f(x,t)+X(x,t)dt$ with $f(x,0)=f_0$ and $f(x,1)=f_1$.   To interpret this homotopy, pick a basepoint $\gamma_0\in M$ and let $\gamma_t$ be the curve in $M$ determined by the time-dependent vector field $X(t)$.  Then, one can follow the value of the function $f$ along $\gamma$.  Dividing the mass equation by density and integrating with respect to time says 
\[f(\gamma_s,s) = f(x_0,0)-\int_0^s \divergence {X(\gamma_t,t)}dt.\] In particular, $f_1(\gamma_1)=f_0(\gamma_0)-\int_0^1 \divergence X(\gamma_t,t)dt$.  The function $f_1$ is the transport of the function $f_0$ along the flow $X$ minus the accrued divergence of $X$.

\subsection{The vorticity equation}
For the remainder of the section, assume that $M$ is dimension three. To generalize the following results to manifolds of dimension $n$, give $\epsilon$ degree $2-n$ and define curl and cross product suitably (see Appendix~\ref{appendix:musical appendix}).  The results are almost identical for dimensions $n\geq 5$, but for $n=2$ and $n=4$ there are subtleties not explored here.

In order to involve collections of homotopy random variables that are not ordinary random variables, a parameter with negative degree needs to be introduced.   Consider the parameter ring $\rR[\epsilon]$, where $\epsilon$ is a degree $-1$ variable. A generic collection of homotopy random variables parametrized by $\rR[\epsilon]$ is an element of the form $f+V\epsilon$ for $f$ a function and $V$ a one-form on $M$. Again, $\epsilon$ squares to zero and so $\exp(f+V\epsilon)=\rho  + \rho V\epsilon$. Then the data must satisfy the condition 
\begin{equation}
\delta(\rho V)=0,\label{vorticity div free eq}
\end{equation}
So the condition on $V$ is that $ (\rho V)^\sharp$ is a divergence free vector field.

Next, look at a homotopy.  A homotopy of $\rR[\epsilon]$-collections has the form
\[
f(t)+X(t)dt + V(t)\epsilon + \pi(t)dt\epsilon.
\]
where $f$, $X$, $V$, and $\pi$ are power series in $t$ with coefficients functions, one-forms, one-forms, and two-forms in $M$, respectively. The exponential is only slightly more complicated than before:
\[
\rho  + \rho Xdt + \rho V\epsilon + \rho (-XV+\pi)dt\epsilon
\]
(the sign in the final term comes from commuting the odd elements $V$ and $dt$)

Then these data must satisfy the mass equation along with Equation~(\ref{vorticity div free eq}) and the following:
\begin{equation}\refstepcounter{equation}
-\partial_t(\rho V) + \delta (\rho \pi - \rho XV)=0.
\label{MC vorticity eq}\tag{\theequation a}
\end{equation}  
\begin{lemma}\label{lemma: constant uniform density}
Every solution $u(t)$ to the vorticity equations for a nonviscous fluid with unit density 
\begin{align*}
\dot{\omega}=(\omega \ip \nabla)u - (u\ip \nabla)\omega
\tag{\theequation b}&& \text{unit density vorticity equation}\label{equation: vorticity incompressible, standard}\\
\divergence{u}=0 && \text{incompressible mass equation}\\
\omega=\curl(u) && \text{definition of $\omega$ as curl of $u$}
\end{align*}

gives rise to a homotopy of $\mathbb{R}[\epsilon]$-collections of homotopy random variables under the identifications $f=0$, $X=u^\flat$, $V=-\star dX$, $\pi=\frac{1}{2}\delta(X\wedge \star X)$. 

Conversely, for every homotopy of $R$-collections of homotopy random variables which satisfies the additional conditions:
\begin{align}
f =0&&\text{constant and uniform density constraint}\\
V =\star X && \text{vorticity constraint}\\
\pi=\tfrac{1}{2}\delta(X\wedge \star X)&&\text{kinetic constraint}
\end{align}
arises from a solution to the vorticity equation with constant unit density and velocity $X^\sharp$.
\end{lemma}
\begin{remark}
In fact, the lemma remains true if the kenetic constraint is replaced with the constraint that $\pi$ be a $\delta$-closed two-form. The kinetic constraint is chosen here because it feeds into the generalization in Lemma~\ref{lemma: compressible Euler}.
\end{remark}
\begin{remark}
In~\cite{Sullivan:3DIFCMEMCAWP3DZVL}, Sullivan analyzed the vorticity equations for a nonviscous fluid with unit density\footnote{Sullivan's equation $\dot{V}=[X,V]$ differs by a sign from the conventional equation.}
and suggested that homotopy probability theory might be helpful in understanding a \emph{finite} models of fluids.  The authors were inspired by this suggestion to explore the link between HPT and fluids and found that there is a connection between HPT and \emph{smooth} models for fluids as well---the connection described in this paper.
\end{remark}
\begin{proof}
Let $u(t)$ be a solution to the vorticity equations and define $(f,X,V,\pi)$ as in the lemma. Then $\rho=\exp(f)=1$ and $\pi$ is $\delta$-closed so that the equations that must be satisfied for this data to be the data of a homotopy are:
\begin{align*}
\delta(X)&=0\\
\delta(V)&=0\\
-\dot{V}-\delta(XV)&=0.
\end{align*}
The mass equation implies the first of these. The second is true by definition of $V$; note also that $-V=\omega^\flat$ under the identifications. Finally, rewrite the right hand side of the vorticity equation, using the fact that both $u$ and $\omega$ are divergence free:
\[
(\omega \ip \nabla)u - (u\ip \nabla)\omega =
(\omega \ip \nabla)u - (u\ip \nabla)\omega+ u(\divergence{\omega})-\omega(\divergence{u})=\curl{(u\times \omega)}.
\]
Then reinterpreting the vorticity equation in light of the identifications of $X$ and $V$, this is
\[
(-\dot{V})^\sharp = \left(-\star d\left(X^\sharp\times V^\sharp\right)^\flat\right)^\sharp=(- \star d \star(X\wedge V))^\sharp=(\delta(XV))^\sharp,
\]
as desired. This argument may be run just as easily in reverse.
\end{proof}
\begin{remark}
There is a variation of this lemma using the same parameter ring without the constant and uniform density assumption. In this case, the relevant equation is not quite the vorticity equation, which is the curl of the Euler equation, but rather the curl of the Euler equation multiplied by the density. This seemingly innocuous change eliminates pressure terms from the resulting equation. 

In this potential variant, some modifications would be necessary to the identification of the forms $(f,X,V,\pi)$ involved in the collection of homotopy random variables in terms of the hydrodynamics variables $(\rho, u)$ because Lemma~\ref{lemma: constant uniform density} has already incorporated certain simplifications that are only possible due to the constant and uniform density assumption. This variant lemma and these modifications will not be stated explicitly here because the variant lemma is a corollary of Lemma~\ref{lemma: compressible Euler} and the appropriate identifications can be seen there.
\end{remark}

\subsection{The Euler equation}
Now consider the differential graded parameter ring $R=\rR[\epsilon,d\epsilon]$. Here $\epsilon$ is again degree $-1$ and $d\epsilon$ is degree $-2$. Continue to assume that the dimension of $M$ is three.

A generic $R$-collection of homotopy random variables is an element of the form 
\begin{equation}\label{general Euler rv}
f+V\epsilon+\sigma d\epsilon + \Psi \epsilon d\epsilon
\end{equation}

 for $f$ a function, $V$ a one-form, $\sigma$ a two-form, and $\Psi$ a three-form on $M$. The exponential is
 \[
 \rho  + \rho V\epsilon + \rho \sigma d\epsilon + \rho (V\sigma+\Psi)\epsilon d\epsilon
  \] (while $d\epsilon$ does not square to zero, the two-form $\sigma$ does). Then these data must satisfy the conditions of Equations~(\ref{MCmasseq})~and~(\ref{vorticity div free eq}) along with the further equations (the coefficients, respectively, of $d\epsilon$, $\epsilon d\epsilon$, and $(d\epsilon)^2$):
\begin{align}
\delta(\rho \sigma) - \rho V&=0
\label{V equation}\\
\delta(\rho V\sigma+\rho \Psi)&=0
\label{trivial equation}\\
-\rho V\sigma-\rho \Psi &= 0.\label{Psi equation}
\end{align}
The equations~(\ref{V equation})~and~(\ref{Psi equation}) determine $V$ and $\Psi$ in terms of $\sigma$ and $f$; then Equation~(\ref{vorticity div free eq}) follows from Equation~(\ref{V equation}) and Equation~(\ref{trivial equation}) follows from Equation~(\ref{Psi equation}), in both cases by applying $\delta$. Then an $R$-collection of homotopy random variables is a pair $(f,\sigma)$ with $f$ an arbitrary function and $\sigma$ an arbitrary $2$-form.

A homotopy between two $R$-collections of homotopy random variables is of the form
\begin{equation*}
f+ Xdt + V\epsilon + \pi dt\epsilon + \sigma d\epsilon + \Phi dtd\epsilon + \Psi \epsilon d\epsilon
\end{equation*}
where as before the variables are power series in $t$ with coefficients in forms of the appropriate degree. The exponential is
\begin{multline*}
\rho + \rho Xdt + \rho V\epsilon + 
\rho (-XV + \pi) dt\epsilon + \rho \sigma d\varepsilon+ 
\rho (X\sigma + \Phi) dtd\epsilon + 
\rho (V\sigma+\Psi) \epsilon d\epsilon
\end{multline*}
(there are no other terms because of the nilpotence of $\Omega^\bullet(M)$).

The equations these variables satisfy are the mass equation (Equation~(\ref{MCmasseq})) as well as Equations~(\ref{vorticity div free eq}) and~(\ref{MC vorticity eq}) above along with (the $t$-dependent versions of) Equations~(\ref{V equation})--(\ref{Psi equation}) above and the following equations, the coefficients of $dt\epsilon d\epsilon$ and $dtd\epsilon$, respectively:
\begin{align}
-\partial_t(\rho(V\sigma+\Psi))&=0;
\label{helicity equation}
\\
\rho (XV-\pi)+\partial_t(\rho \sigma)+\delta(\rho X\sigma + \rho \Phi)&=0.
 \refstepcounter{equation}\tag{\theequation a}\label{MC momentum equation}
\end{align}
However, several of these equations follow formally from one another. 
\begin{itemize}
\item Equation~(\ref{vorticity div free eq}) arises by applying $\delta$ to Equation~(\ref{V equation}), 
\item Equation~(\ref{MC vorticity eq}) arises by applying $\delta$ to Equation~(\ref{MC momentum equation}) and then using Equation~(\ref{V equation}) to replace the time derivative, and
\item Equations~(\ref{trivial equation}) and~(\ref{helicity equation}) arises by applying $\delta$ and $\partial_t$, respectively, to Equation~(\ref{Psi equation}).
\end{itemize}
Recall that as above, $V$ and $\Psi$ are determined by Equations~(\ref{V equation}) and~(\ref{Psi equation}); similarly, Equation~\ref{MC momentum equation} determines $\pi$. 
 
Then after simplification, a homotopy between two $R$-collections of random variables is the data of a collection $(f, X, \sigma, \Phi)$ which satisfy the mass equation.
\begin{remark}
Because a homotopy between two $R$-collections of random variables must satisfy only the mass equation and has two free variables $\sigma$ and $\Phi$, it may seem strange to try to relate it to fluid flow. That is, by varying the interpretation of the free variables, this could encapsulate any collection of variables where some part satisfies the mass equation. However, the equation seems fairly well-suited to the momentum equation in that the variables are generally quantities of independent physical interest. Perhaps a more judicious choice of parameter ring (in particular, one which is not contractible) might yield a version of the equation which does not have this underdetermination. On the other hand, it is also possible that there is a further natural algebraic condition to put on the homotopy which imposes the other conditions in the following lemma.
\end{remark}
\begin{lemma}\label{lemma: compressible Euler}
Every solution $(\rho,u,p)$ to the compressible Euler equations 
\begin{align*}
\dot{\rho}+\divergence{(\rho u)}=0&&\text{mass equation}\\
\dot{u} + \nabla\left(\frac{u\ip u}{2}\right) -u\times(\curl{u}) + \frac{\nabla p}{\rho}=0\tag{\theequation b}\label{momentum euler equation alternate}&&\text{Euler momentum equation}
\end{align*}
gives rise to a homotopy of $R$-collections of homotopy random variables under the identifications $\rho=\rho$, $X=u^\flat$, $V=-\frac{\star d(\rho X)}{\rho}$, $\pi=\frac{1}{2}\delta(X\wedge \star X)-\delta(X)\star X$, $\sigma=\star X$, $\Phi=\frac{\star p}{\rho}$, and $\Psi=X\wedge dX$.

Conversely, for every homotopy of $R$-collections of homotopy random variables which satisfies the additional conditions:
\begin{align}
\sigma =\star X && \text{velocity constraint}\\
\pi=\tfrac{1}{2}\delta(X\wedge \star X)-\delta(X)(\star X)&&\text{modified kinetic constraint}
\end{align}
arises from a solution to the Euler equations with density $\rho$, velocity $X^\sharp$, and pressure $\star (\rho\Phi)$.
\end{lemma}
\begin{proof}
The verification is an exercise, recorded here for convenience.

Suppose that $(\rho, u,p)$ is a solution to the compressible Euler equations. It is straightforward to check that the choice for $V$ satisfies the identity of Equations~(\ref{V equation}). To see that the choice for $\Psi$ satisfies Equation~(\ref{Psi equation}), one verifies
\[
-V\sigma = \frac{\star d(\rho X)}{\rho}\wedge \star X = \frac{1}{\rho}d(\rho X)\wedge X= \frac{d\rho}{\rho}\wedge X\wedge X + dX\wedge X = dX\wedge X.
\]
Finally, to see that the choice for $\pi$ satisfies Equation~(\ref{MC momentum equation}), one can first rewrite Equation~(\ref{MC momentum equation}) by using the Batalin--Vilkovisky relation (see Appendix~\ref{appendix:musical appendix}) as follows.
\begin{align*}
\delta(\rho X\sigma)&=\delta(\rho X)\sigma +\rho \delta(X\sigma)-X\delta(\rho\sigma) -\rho \delta(X)\sigma +\rho X\delta(\sigma)\\
&=
\delta(\rho X)\star X +\rho \delta(X\wedge \star X)-X\rho V -\rho \delta(X)\star X -\rho X\wedge \star d X
\\
&=
\rho(\pi-XV) + 
\delta(\rho X)\star X +\frac{\rho}{2} \delta(X\wedge \star X)-\rho X\wedge \star d X
.
\end{align*}
Then the left hand side of Equation~(\ref{MC momentum equation}) is equivalent to
\begin{multline*}
\quad{} \partial_t(\rho \sigma)+\delta(\rho \Phi)+\frac{\rho}{2} \delta(X\wedge \star X)-\rho X\wedge \star d X
\\=
(\dot{\rho}+\delta(\rho X))\star X + \rho\left (\star\dot{X} -X\wedge \star dX +\frac{1}{2} \delta(X\wedge \star X) +\star \frac{dp}{\rho}\right ).
\end{multline*}
The term $\dot{\rho}+\delta(\rho X)$ in this sum vanishes by Equation~(\ref{MCmasseq}). Then since $\rho$ is a unit and $\star$ an isomorphism, it remains only to be seen that the equation
\[
\dot{X} -\star(X\wedge \star dX) +\frac{1}{2} d\star(X\wedge \star X) + \frac{dp}{\rho}
\]
is equivalent to the Euler momentum equation. But the former is also
\[
(\dot{u})^\flat - \left(X^\sharp \times(\curl{X^\sharp})\right)^\flat + \left(\frac{1}{2}\nabla(X\ip X)\right)^\flat+\left(\frac{\nabla p}{\rho}\right)^\flat,
\]
which is the momentum equation up to the musical isomorphism.

This argument can be run in reverse for a given homotopy satisfying the constraints above.
\end{proof}

\section{Applications of the HPT framework to fluid flow}\label{section: applications}
The paper concludes with some speculative directions for future work.

\subsection{Homotopy statistics}
The homotopy probability theory framework provides some new ideas for studying fluid flow.  One idea is to use the homotopy statistics afforded by homotopy probability theory which yield hierarchies of time-independent numerical invariants.  For any collection $\mathcal{X}$ of homotopy random variables parametrized by $R$, the expected value of the moment generating function $E\left(\exp(\mathcal{X}\right))\in R$ is the complete set of homotopy statistics of the collection $\mathcal{X}$:  if $\mathcal{X}=\sum_i s_iX_i$, then the joint moment $E\left(X^{r_1}_{i_1}\cdots X^{r_k}_{i_k}\right)$ is determined by the coefficient of $s^{r_1}_{i_1}\cdots s^{r_k}_{i_k}$ in $E\left(\exp(\mathcal{X}\right))\in R$.
The fundamental theorem of homotopy probability theory implies that homotopic collections of homotopy random variables have statistics that differ by exact elements of the parameter ring.  In particular, if $R$ has no cohomology, then homotopic collections of homotopy random variables have exactly the \emph{same} statistics.  

More explicitly, suppose $(\rho,u,p)$ is a solution to the compressible Euler equations on a compact Riemannian manifold $M$ and $H(t,dt)=f(t)+X(t)dt+\cdots + \Psi(t) \epsilon d\epsilon$ is the corresponding homotopy of collections of homotopy random variables from Lemma~\ref{lemma: constant uniform density}.  The Fundamental Theorem implies that the statistics of $H(t_1,0)$ and $H(t_2,0)$ have the same cohomology.

The statistics depend on the expectation map.  In the $\rR$-valued homotopy probability space associated to $M$, the expectation map, and more generally any multilinear map built out of the expectation and the product on $\Omega^\bullet(M)$, vanishes outside degree zero.  Then the statistics of a parameterized collection of homotopy random variables only has access to the degree zero number \[E\exp(f(t)) =\int_M \rho(t) dV. \]
This is the integral of the density over the manifold, which yields the mass of the fluid, a time independent quantity.  However, the homotopy statistics with values in \emph{any} pre-homotopy probability space with zero differential will produce time independent invariants of fluids.  Perhaps investigating the statistics for other expectation maps valued in possibly other pre-probability spaces (such as the cohomology valued expectation) would assist with problem described by Arnold and Khesin in~\cite[p. 176]{ArnoldKhesin:TMH}: \begin{quote}\emph{The dream is to define\ldots a hierarchy of invariants for generic vector fields such that, whereas all the invariants of order $\le k$ have zero value for a given field and there exists a nonzero invariant of order $k+1$, this nonzero invariant provides a lower bound for the field eneregy.}\end{quote}
For example, the three-form $\Psi$ that appears in one of the terms in a generic collection of homotopy random variables parametrized by $\rR[\epsilon,d\epsilon]$.  For the collections of homotopy random variables coming from solutions to the Euler equation, $\int_M \Psi$ computes the helicity of the fluid ~\cite{ArnoldKhesin:TMH}.
\subsection{Change of parameter ring}
As indicated in Section~\ref{section:hrv on a Riemannian manifold}, there is room for improvement in the choice of parameter ring. Ideally, a clever choice would eliminate some of the necessary constraints, incorporate other features like energy dynamics or viscosity, or provide a deeper conceptual understanding. One note is that while all the parameter rings considered here are differential graded algebras, one could also imagine choosing a parameter ring which is itself not a differential graded algebra but rather a homotopy pre-probability space (there is a tensor product of homotopy pre-probability spaces).

\subsection{Finite models}
Another way that homotopy probability may contribute to fluid flow is to produce new combinatorial models.   Indeed, this idea was proposed in \cite{Sullivan:3DIFCMEMCAWP3DZVL} and was the idea that first motivated this paper.
A  finite cochain model for a Riemannian manifold $M$ gives rise to a homotopy probability space quasi-isomorphic to the smooth homotopy probability space associated to $M$.  What might be called ``the second fundamental theorem of HPT'' is the fact that collections of homotopy random variables can be transported naturally along morphisms.  Therefore, homotopies between collections of homotopy random variables in the finite cochain model, which are solutions to finite dimensional ordinary differential equations and so always exist, can be transported to a homotopy of collections of homotopy random variables in the smooth model.  At this level of generality, such tranpsorted homotopies need not satisfy the constraints of the various lemmas of Section~\ref{section:hrv on a Riemannian manifold} and therefore need not satisfy any version of the Euler equation.  However, one could study combinatorial versions of these constraints (like a combinatorial version of curl \cite{Wilson:DFFFM,Sullivan:3DIFCMEMCAWP3DZVL}).  An understanding of how the combinatorial constraints relate to the smooth constraints under transport would yield the ingredients for both new finite models of fluid flow as well as a theoretical tool to use finite methods to study smooth fluid flow.

\appendix
\section{Vector fields and differential forms}\label{appendix:musical appendix} 
Let $M$ be a finite volume Riemannian manifold of dimension $n$, let $TM$ denote the tangent 
bundle of $M$, and let $V^k(M):=\Gamma(\Lambda^k TM)$ denote the 
degree $k$ multivector fields on $M$. The metric $g$ on $M$ defines a bundle 
isomorphism between $TM$ and $T^\star{}M$ which extends to the exterior 
products and then to sections. Sometimes these isomorphisms are 
called the musical isomorphisms and are denoted 
\begin{equation} 
\label{musical} 
\xymatrix{
	V^k(M)\ar@<.6ex>[r]^\flat & \Omega^k(M)\ar@<.6ex>[l]^\sharp
}
\end{equation} 
% \[
% \begin{tikzcd} 
% V^k(M) \ar[shift left]{r}{\flat} & \Omega^k(M) \ar[shift left]{l}{\sharp} 
% \end{tikzcd} 
% \]
Familiar operators from multivariable calculus correspond to familiar operators in differential geometry. All of the following can either be taken as definitions in the Riemannian context for arbitrary dimension $n$ and degrees $i$ and $j$ or verified in a local coordinate chart with any standard definition in $\rR^3$, where by choosing $i=j=1$ throughout one restricts to the usual operators on and between functions and vector fields alone. 
\begin{itemize}
\item The inner product on one-forms is defined as $\langle X^\flat, Y^\flat\rangle=g(X,Y)$. 
\item This inner product is extended linearly over functions to $i$-forms. 
\item The top forms are one dimensional and the choice of orientation on $M$ gives a preferred unit $n$-form $dV$. 
\item The Hodge star operator $\star: \Omega^i(M)\to \Omega^{n-i}(M)$ is defined by $\alpha \wedge \star \beta = \langle \alpha,\beta\rangle dV$.
\item The codifferential $\delta:\Omega^{i}(M)\to \Omega^{i-1}(M)$ is defined as $(-1)^{ni+1}\star d\star$.
\item The divergence $\divergence{}:V^i(M)\to V^{i-1}(M)$ corresponds to the codifferential $\delta:\Omega^i(M)\to\Omega^{i-1}(M)$. That is, $\divergence{X}=\delta(X^\flat)$.
\item The gradient $\nabla:V^{i-1}(M)\to V^{i}(M)$ corresponds to the de Rham differential $d$. That is, $\nabla f =  (df)^\sharp$.
\item The curl $\curl{}:V^i(M)\to V^{n-i-1}(M)$ corresponds to the operator $\star{}d$. That is, $\curl{X}= \left(\star d(X^\flat)\right)^\sharp.$
\item The cross product $V^i(M)\otimes V^j(M)\to V^{n-i-j}(M)$ corresponds to the operator $\star\wedge$. That is, $X\times Y=\left(\star(X^\flat \wedge  Y^\flat)\right)^\sharp$.
\end{itemize}
There is one further property that will be used which describes the compatibility between the codifferential $\delta$ and the wedge product on forms. Unlike the de Rham differential $d$, which is a derivation (that is, a first order differential operator), the codifferential is a \emph{second} order differential operator. Then the equation it satisfies for arbitrary homogeneous forms $\alpha$, $\beta$, and $\gamma$ is
\begin{align*}
\delta(\alpha\wedge \beta\wedge \gamma)&=\delta(\alpha\wedge\beta)\wedge \gamma + (-1)^{|\alpha|}\alpha\wedge \delta(\beta\wedge \gamma) + (-1)^{|\beta||\gamma|}\delta(\alpha\wedge \gamma)\wedge \beta \\&\quad{}- \delta(\alpha)\wedge\beta\wedge\gamma -(-1)^{|\alpha|}\alpha\wedge \delta(\beta)\wedge \gamma -(-1)^{|\alpha|+|\beta|}\alpha\wedge \beta\wedge\delta(\gamma)
\end{align*}
as can be verified directly.

\bibliography{references-2016}

\providecommand{\bysame}{\leavevmode\hbox to3em{\hrulefill}\thinspace}
\providecommand{\MR}{\relax\ifhmode\unskip\space\fi MR }
\providecommand{\MRhref}[2]{%
  \href{http://www.ams.org/mathscinet-getitem?mr=#1}{#2}
}
\providecommand{\href}[2]{#2}
\begin{thebibliography}{DCPT15}

\bibitem[AFW06]{ArnoldFalkWinther:FEECHTA}
Douglas~N. Arnold, Richard~S. Falk, and Ragnar Winther, \emph{Finite element
  exterior calculus homological techniques and applications}, Acta Numer.
  \textbf{15} (2006), 1--155.

\bibitem[AK98]{ArnoldKhesin:TMH}
Vladimir~I. Arnol'd and Boris~A. Khesin, \emph{Topological methods in
  hydrodynamics}, Appl. Math. Sci., Springer, 1998.

\bibitem[Are46]{Arens:TSLOCTR}
Richard Arens, \emph{The space ${L}^\omega$ and convex topological rings},
  Bull. Amer. Math. Soc. \textbf{52} (1946), no.~10, 931--935.

\bibitem[DCPT15]{DrummondColeParkTerilla:HPTII}
Gabriel~C. Drummond-Cole, Jae-Suk Park, and John Terilla, \emph{Homotopy
  probability theory {I}{I}}, J. Homotopy Relat. Struct. \textbf{10} (2015),
  623--635.

\bibitem[DCT14]{DrummondColeTerilla:CIHPT}
Gabriel~C. Drummond-Cole and John Terilla, \emph{Cones in homotopy probability
  theory}, arXiv:1410.5506, 2014.

\bibitem[DP15]{DotsenkoPoncin:TTH}
Vladimir Dotsenko and Norbert Poncin, \emph{A tale of three homotopies}, Appl.
  Categ. Structures (2015), 1--29.

\bibitem[GJF12]{GwilliamJohnsonFreyd:HDFDFDIDBVF}
Owen Gwilliam and Theo Johnson-Freyd, \emph{How to derive {F}eynman diagrams
  for finite-dimensional integrals directly from the {B}{V} formalism},
  arXiv:1202.1554, 2012.

\bibitem[Gro15]{Gromov:SPESLM2014}
Misha Gromov, \emph{Symmetry, probability, entropy: synopsis of the lecture at
  {M}{A}{X}{E}{N}{T} 2014}, Entropy \textbf{17} (2015), 1273--1277.

\bibitem[Par11]{Park:L}
Jae-Suk Park, \emph{Einstein chair lecture}, City University of New York,
  November 2011.

\bibitem[Par15]{Park:HTPSICIHLA}
\bysame, \emph{Homotopy theory of probability spaces {I}: Classical
  independence and homotopy {L}ie algebras}, arXiv:1510.08289, 2015.

\bibitem[PP13]{ParkPark:EHTPISPH}
Jae-Suk Park and Jeehoon Park, \emph{Enhanced homotopy theory for period
  integrals of smooth projective hypersurfaces}, arXiv:1310.6710, 2013.

\bibitem[Sul14]{Sullivan:3DIFCMEMCAWP3DZVL}
Dennis Sullivan, \emph{3{D} incompressible fluids: Combinatorial models,
  eigenspace models, and a conjecture about well-posedness of the 3{D} zero
  viscosity limit}, J. Differential Geom. \textbf{97} (2014), no.~1, 141--148.

\bibitem[Tao10]{Tao:254AN5FP}
Terence Tao, \emph{254{A}, notes 5: free probability}, blog post,
  \url{https://terrytao.wordpress.com/2010/02/10/245a-notes-5-free-probability/},
  2010.

\bibitem[Wil11]{Wilson:DFFFM}
Scott Wilson, \emph{Differential forms, fluids, and finite models}, Proc. Amer.
  Math. Soc. \textbf{139} (2011), 2597--2604.

\end{thebibliography}
\bibliographystyle{amsalpha}
\end{document}